\documentclass[11pt]{amsart}

\usepackage{graphicx, pinlabel}
\newtheorem{theorem}{Theorem}[section]
\newtheorem{proposition}[theorem]{Proposition}
\newtheorem{lemma}[theorem]{Lemma}
\newtheorem{corollary}[theorem]{Corollary}
\newtheorem*{maxrootcorollary}{Corollary~\ref{coro:maxroot}}
\newtheorem*{almostallcorollary}{Corollary~\ref{coro:almostall}}

\theoremstyle{definition}

\numberwithin{equation}{section}

\newcommand{\longpage}{\enlargethispage{\baselineskip}}

\newcommand{\Mod}{\operatorname{Mod}}
\newcommand{\lcm}{\operatorname{lcm}}
\renewcommand{\O}{\operatorname{{\mathcal O}}}

\begin{document}

\title[Roots of Dehn twists]
{Roots of Dehn twists}

\author{Darryl McCullough}
\address{Department of Mathematics\\
University of Oklahoma\\
Norman, Oklahoma 73019\\
USA} 
\email{dmccullough@math.ou.edu}
\urladdr{www.math.ou.edu/$_{\widetilde{\phantom{n}}}$dmccullough/}
\thanks{The first author was supported in part by NSF grant DMS-0802424}

\author{Kashyap Rajeevsarathy}
\address{Department of Mathematics\\
University of Oklahoma\\
Norman, Oklahoma 73019\\
USA} 
\email{kashyap@math.ou.edu}

\date{\today}


\begin{abstract}
D. Margalit and S. Schleimer found examples of roots of the Dehn twist
$t_C$ about a nonseparating curve $C$ in a closed orientable surface, that
is, homeomorphisms $h$ such that $h^n=t_C$ in the mapping class group.  Our
main theorem gives elementary number-theoretic conditions that describe the
$n$ for which an $n^{th}$ root of $t_C$ exists, given the genus of the
surface. Among its applications, we show that $n$ must be odd, that the
Margalit-Schleimer roots achieve the maximum value of $n$ among the roots
for a given genus, and that for a given odd $n$, $n^{th}$ roots exist for
all genera greater than $(n-2)(n-1)/2$. We also describe all $n^{th}$ roots
having $n$ greater than or equal to the genus.
\end{abstract}

\maketitle

A natural question about mapping class groups is whether a Dehn twist has a
root. That is, given a Dehn twist $t_C$ about a simple closed curve $C$ in
a closed orientable surface $G$ and an integer degree $n>1$, does there
exist an orientation-preserving homeomorphism $h$ with $h^n=t_C$ in the
mapping class group $\Mod(G)$? It is easy to find examples of roots when
$C$ is a separating curve, but for a nonseparating curve it is not
immediately apparent that roots exist. Note that since all nonseparating
curves are equivalent under homeomorphisms of $G$, the question is
independent of the particular curve used.

Recently some beautiful examples of such roots were found by D. Margalit
and S. Schleimer~\cite{MS}. They constructed roots of degree $2g+1$ for the
Dehn twist $t_{g+1}$ about a nonseparating curve in the surface of genus
$g+1\geq 2$. We will describe those examples from our viewpoint in
Section~\ref{sec:Margalit-Schleimer}, after stating and proving our main
result Theorem~\ref{thm:main} in Section~\ref{sec:main}.

Theorem~\ref{thm:main} says that given $g$ and $n$, $t_{g+1}$ has a root of
degree $n$ if and only if there exists a collection of integers satisfying
certain equations. In fact, the conjugacy classes of roots correspond to
the solutions. Its proof is an exercise in the well-studied theory of
group actions on surfaces. We present it using the language of orbifolds
(see W. Thurston~\cite[Chapter 13]{Thurston}), rather than the classical
description that uniformizes the action and then works with the lifted
isometries of the hyperbolic plane.

A number of applications can be obtained from Theorem~\ref{thm:main} by
elementary considerations. An immediate consequence is
\begin{maxrootcorollary}
Suppose that $t_{g+1}$ has a root of degree $n$. Then
\begin{enumerate}
\item[(a)] $n$ is odd.
\item[(b)] $n\leq 2g+1$.
\end{enumerate}
\end{maxrootcorollary}
\noindent Thus the Margalit-Schleimer roots always have the maximum degree
among the roots of $t_{g+1}$ for a given genus.

Section~\ref{sec:genus_set} concerns the set of $g\geq 0$ for which
$t_{g+1}$ has a root of degree a fixed $n$. This set always contains all
but at most $(n-1)^2/4$ values, those in a set $T(n)$ which is easy to
describe, and whose maximum element is $n(n-3)/2$:
\begin{almostallcorollary} 
For $n$ odd, $t_{g+1}$ has a root of degree $n$ whenever $g\notin T(n)$.
Consequently, $t_{g+1}$ has a root of degree $n$ for all
$g+1>(n-2)(n-1)/2$.\par
\end{almostallcorollary}
\noindent For prime $n$, $t_{g+1}$ has no root of degree $n$ exactly when
$g\in T(n)$, but when $n$ is not prime, roots of degree $n$ may occur
when $g\in T(n)$.

In general, it is hard to use Theorem~\ref{thm:main} to work out the full
set of degrees of roots of $t_{g+1}$ for a given genus, although our
results allow easy computation of the possible prime degrees of roots.
Roots of large degree are rather limited, however, and in
Theorem~\ref{thm:largeroots} of Section~\ref{sec:de} we describe the set of
roots having degree $n\geq g$. Curiously, there is a root of degree $g$
only when $g+1=4$. For $g\geq 2$ and degrees satisfying $g+1\leq n < 6(g +
2)/5$, all roots are of a restricted type that we call $(d,e)$-roots, and
the only larger degree is that of the Margalit-Schleimer roots, $2g+1$. The
genera for which $t_{g+1}$ has a $(d,e)$-root of a given degree $n$ are
easily found from a prime factorization of $n$, and the $n$ for which a
given genus has a $(d,e)$-root of degree $n$ can be calculated on a desktop
computer using GAP~\cite{GAP} for genera up to~1,000,000.

As shown in Figure~\ref{fig:rootset}, we can combine these results to get a
sense of the set of pairs $(g,n)$ for which there is a root of degree $n$
for $t_{g+1}$.
\begin{figure}
\labellist
\small
\pinlabel $n$ [B] at -1 310
\pinlabel $30$ [B] at -9 276
\pinlabel $5$ [B] at -5 50
\pinlabel $n=2g+1$ [B] at 140 320
\pinlabel $n=6(g+2)/5$ [B] at 240 320
\pinlabel $n=g$ [B] at 320 320
\pinlabel $g=(n-2)(n-1)/2$ [B] at 365 118
\pinlabel $g$ [B] at 442 -2
\pinlabel $5$ [B] at 54 -9
\pinlabel $45$ [B] at 414 -9
\endlabellist
\centering
\includegraphics[width=65 ex]{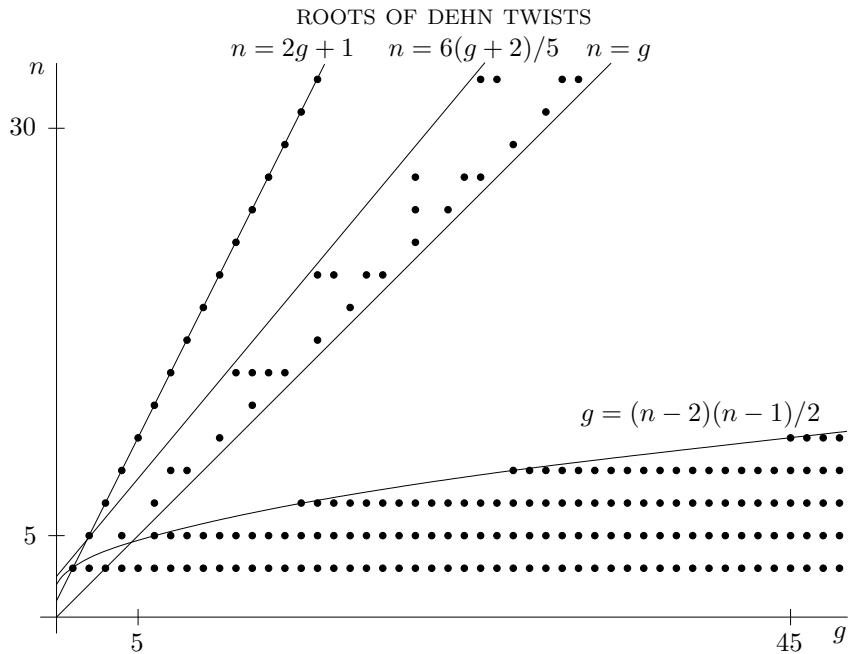}
\caption{Some of the $(g,n)$ pairs in the rectangle $[0,48]\times [0,33]$
for which $t_{g+1}$ has a root of degree $n$. The Margalit-Schleimer
pairs lie on $n=2g+1$. Below this is a region $6(g+2)/5\leq n<2g+1$ with
no pairs, then the region $g<n<6(g+2)/5$ of $(d,e)$-root pairs (except
for $(1,3)$, which is a Margalit-Schleimer pair). The $(d,e)$-root pairs
shown here are accurate. The pair $(3,3)$ is the only element of the
degree set that lies on $n=g$. In the stable region $3\leq n$ and
$(n-2)(n-1)/2\leq g$, every pair $(g,n)$ with $n$ odd occurs.  In the
region above $g=(n-2)(n-1)/2$, the primary roots (see
Section~\ref{sec:genus_set}) asymptotically give about half of the pairs
with $n$~odd.}
\label{fig:rootset}
\end{figure}

Our definition of roots requires them to be orientation-preserving, but
this restriction is not necessary. In
Section~\ref{sec:orientation-reversing}, we check that $t_{g+1}^\ell$ can
be isotopic to $h^n$ with $h$ orientation-reversing only when $\ell = 0$.
We also observe that roots can only be conjugate by orientation-preserving
homeomorphisms. Thus Theorem~\ref{thm:main} is a complete classification of
all roots of $t_{g+1}$ in the homeomorphism group, up to conjugacy.

Theorem~\ref{thm:main} gives some information on roots of powers of
$t_{g+1}$, that is, on fractional powers of $t_{g+1}$, as we discuss in
Section~\ref{sec:fractional}. For example, $t_2^2$ has a fourth root
although as we saw in Corollary~\ref{coro:maxroot}, $t_2$ does not have a
square root. Our methods can also be used to understand the roots of Dehn
twists about separating curves. Of course in this case, the roots will
depend on the genera of the complementary components. We expect to pursue
these ideas in future work.

\newpage

\section{The main theorem}
\label{sec:main}

For us, a Dehn twist means a left-handed Dehn twist, one for which the
image of an arc crossing $C$ turns to the left approaching $C$, as seen
from the outside of the oriented surface.

By a \textit{data set} we mean a tuple
$(n,g_0,(a,b);(c_1,n_1),\ldots,(c_m,n_m))$ where
$n$, $g_0$, $a$, $b$, the $c_i$ and the $n_i$ are integers satisfying
\begin{enumerate}
\item[(i)] $n>1$, $g_0\geq 0$, each $n_i>1$, and each $n_i$ divides $n$,
\item[(ii)] $\gcd(a,n) = \gcd(b,n)= 1$ and each $\gcd(c_i,n_i)=1$,
\item[(iii)] $a+b=ab\bmod n$, and
\item[(iv)] $a + b + \displaystyle\sum_{i=1}^m \dfrac{n}{n_i}\,c_i = 0\bmod n$.
\end{enumerate}
\noindent By condition (ii), $a$ and $b$ are units $\bmod\ n$, so condition
(iii) requires $n$ to be odd, and conditions (iii) and (iv) require $m\geq
1$. The number $n$ is called the \textit{degree} of the data set, and the
positive integer $g$ defined by
\[ g=g_0n + \frac{1}{2}\sum_{i=1}^m \frac{n}{n_i}(n_i-1) \]
is its \textit{genus.} Note that $g$ is independent
of the values of $a$, $b$, and the $c_i$, and no data set has
genus~$0$. Later we will check that $n\leq 2g+1$.

We consider two data sets to be the same if they differ by interchanging
$a$ and $b$, changing $a$ or $b\bmod n$, changing a $c_i\bmod n_i$,
or reordering the pairs $(c_1,n_1)\ldots\,$, $(c_m,n_m)$. With this
understanding, we have our main result.

\begin{theorem} For a given $n>1$ and $g\geq 0$, data sets of genus $g$ and
degree $n$ correspond to the conjugacy classes in $\Mod(G_{g+1})$ of the
roots of $t_{g+1}$ of degree~$n$.\par
\label{thm:main}
\end{theorem}
\begin{proof}
We will first prove that every conjugacy class of roots of degree $n$
yields a data set of degree $n$ and genus $g$.

Fix a nonseparating curve $C$ in an oriented surface $G$ of genus
$g+1$. Choose a closed tubular neighborhood $N$ of $C$, and put
$F_0=\overline{G-N}$. By isotopy we may assume that $t_C(C)=C$, $t_C(N)=N$,
and $t_C\vert_{F_0}=id_{F_0}$.

Suppose that $h$ is a root of $t_C$ of degree $n$. We have
$t_C=ht_Ch^{-1}=t_{h(C)}$, which implies that $h(C)$ is isotopic to
$C$. Changing $h$ by isotopy, we may assume that $h$ preserves $C$ and
takes $N$ to $N$. Put $h_0=h\vert_{F_0}$.

Since $h^n\simeq t_C$ and both preserve $C$, there is an isotopy from $h^n$
to $t_C$ preserving $C$ and hence one taking $N$ to $N$ at each time. That
is, $h_0^n$ is isotopic to $id_{F_0}$. By the Nielsen-Kerckhoff theorem,
$h_0$ is isotopic to a homeomorphism whose $n^{th}$ power is $id_{F_0}$.
(The Nielsen-Kerckhoff Theorem was proven in general by S. Kerckhoff
\cite{K1, K2}. The cyclic case we need here was given by J. Nielsen in
\cite{N}, although W. Fenchel \cite{Fenchel1, Fenchel2} gave the first
complete proof. See the introduction in H. Zieschang's monograph~\cite{Z}.)
So we may change $h$ by isotopy so that $h_0^n=id_{F_0}$.

We cannot have $h_0^r=id_{F_0}$ for any $r$ with $1<r<n$. For a minimal
such $r$ would have to divide $n$, and then $h^r$ would be isotopic either
to the identity or to some power $t_C^\ell$ of $t_C$, forcing
$t_C=h^n=t_C^{\ell n/r}$ with $\ell n/r$ either $0$ or greater than $1$. So
$h_0$ defines an effective action of the cyclic group $C_n$ of order $n$ on
$F_0$. Filling in the two boundary circles of $F_0$ with disks and
extending $h_0$ to a homeomorphism $t$ by coning, we obtain a $C_n$-action
on the closed orientable surface $F$ of genus~$g$, where $C_n=\langle
t\;\vert\;t^n=1\rangle$.

Later, we will show that $h$ cannot interchange the sides of $C$. For now,
assume that it does not. Under this assumption, $t$ fixes the center points
$P$ and $Q$ of the two disks of $\overline{F-F_0}$. The orientation of $G$
determines one for $F_0$ and hence for $F$, so we may speak of directed
angles of rotation about $P$ and $Q$ (and any other fixed points of
$t$). The rotation angle of $t$ at $P$ is $2\pi k/n$ for some $k$ with
$\gcd(k,n)=1$. As illustrated in Figure~\ref{fig:partial_twist}, the
rotation angle at $Q$ must be $2\pi(1-k)/n$, in order that $h^n$ be a
single Dehn twist.
\begin{figure}
\labellist
\small
\pinlabel $A$ [B] at 160 95
\pinlabel $P$ [B] at 73 80
\pinlabel $t(A)$ [B] at -10 142
\pinlabel $B$ [B] at 358 95
\pinlabel $t(B)$ [B] at 250 10
\pinlabel $Q$ [B] at 260 105
\pinlabel $A$ [B] at 555 175
\pinlabel $B$ [B] at 555 30
\endlabellist
\centering
\includegraphics[width=65 ex]{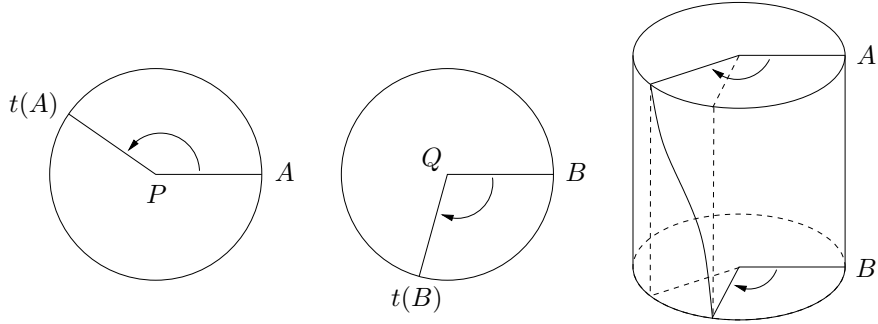}
\caption{The local effect of $t$ on disk neighborhoods of $P$ and $Q$ in
$F$, and the effect of $h$ on the neighborhood $N$ of $C$ in $G$. Only the
boundaries of the disk neighborhoods are contained in $G$, where they form
the boundary of $N$. The condition $a+b=ab\bmod n$ holds when
the angle at $P$ is $2\pi a^{-1}/n$ and the angle at $Q$ is
$2\pi(1-a^{-1})/n$.}
\label{fig:partial_twist}
\end{figure}

\begin{figure}
\labellist
\small
\pinlabel $p$ [B] at 272 68
\pinlabel $q$ [B] at 225 92
\pinlabel $x_0$ [B] at 152 14
\pinlabel $x_1$ [B] at 180 95
\pinlabel $x_2$ [B] at 126 95
\pinlabel $x_3$ [B] at 72 95
\pinlabel $\alpha$ [B] at 312 102
\pinlabel $\beta$ [B] at 258 138
\pinlabel $\gamma_1$ [B] at 182 152
\endlabellist
\centering
\includegraphics[width=75 ex]{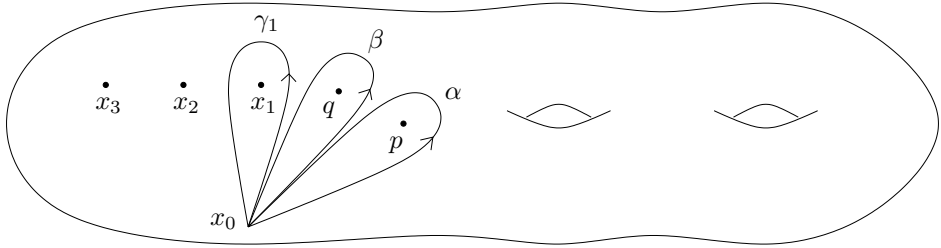}
\caption{A quotient orbifold $\O$, for the case $m=3$ and $g_0=2$.}
\label{fig:orbifold}
\end{figure}
Now, let $\O$ be the quotient orbifold for the action of $C_n$ on $F$, and
let $g_0$ be the genus of the underlying $2$-manifold $\vert \O\vert$.
Figure~\ref{fig:orbifold} shows $\O$, with cone points $p$ and $q$ of order
$n$ (the images of the points $P$ and $Q$ of $F$) and possibly other cone
points $x_1,\ldots\,$, $x_m$ of some orders $n_1\ldots\,$, $n_m$. The
figure also shows some of the generators $\alpha$, $\beta$, and $\gamma_1$
of $\pi_1^{orb}(\O)$. Along with similar generators $\gamma_i$ going around
the other $x_i$ and standard generators $a_j$ and $b_j$, $1\leq j\leq g_0$
in the ``surface part'' of $\O$, we have a presentation
\begin{gather*} 
\pi_1^{orb}(\O)=\langle \alpha, \beta, \gamma_1,\ldots, \gamma_m,
a_1,b_1,\ldots, a_{g_0}, b_{g_0}\;\vert\;\\
\alpha^n=\beta^n=\gamma_1^{n_1}=\cdots =\gamma_m^{n_m}=1,\;
\alpha \beta \gamma_1 \cdots \gamma_m=\prod_{j=1}^{g_0}[a_j,b_j]\;\rangle
\end{gather*}

\begin{figure}
\labellist
\small
\pinlabel $p$ [B] at 72 62
\pinlabel $\widetilde{x_0}$ [B] at 158 68
\pinlabel $\widetilde{\alpha}$ [B] at 150 90
\pinlabel $t(\widetilde{x_0})$ [B] at 47 148
\pinlabel $t^a(\widetilde{x_0})$ [B] at 155 110
\endlabellist
\centering
\includegraphics[width=25 ex]{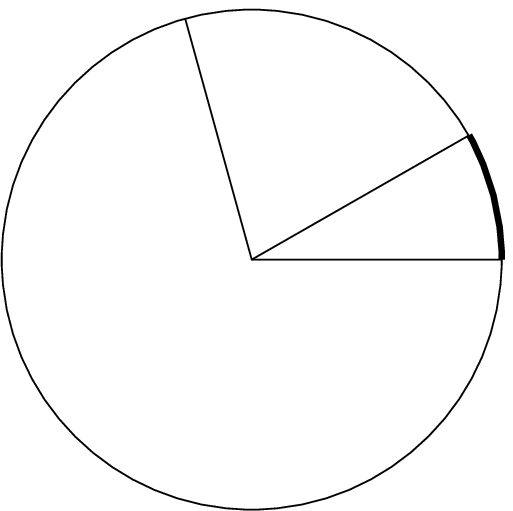}
\caption{The lift of $\alpha$ to $F$ starting at $\widetilde{x_0}$.}
\label{fig:lift}
\end{figure}

From orbifold covering space theory, the orbifold covering map $F\to \O$
corresponds to an exact sequence
\[ 1 \longrightarrow \pi_1(F) \longrightarrow \pi_1^{orb}(\O)
\stackrel{\rho}{\longrightarrow} C_n \longrightarrow 1\ .\] Here, $C_n$ is
the group of covering transformations, generated by $t$, and $\rho$ is
obtained by lifting path representatives of elements of
$\pi_1^{orb}(\O)$--- these do not pass through the cone points so the lifts
are uniquely determined. To find $\rho(\alpha)$, we note first that the
loop $\alpha$ lifts as shown in Figure~\ref{fig:lift}, so $\rho(\alpha)$
maps to $t^a$ where $t^a$ has rotation angle $2\pi/n$ about $P$. Since $t$
acts with rotation angle $2\pi k/n$, we have $ka= 1\bmod n$ so $k=a^{-1}
\bmod n$. Similarly at $Q$, the rotation angle of $t$ is $2\pi b^{-1}/n$.
Since $h^n=t_{g+1}$, the left-hand twisting angle along $N$ in
Figure~\ref{fig:partial_twist} is $2\pi/n$. This requires $2\pi
b^{-1}n-(-2\pi a^{-1}/n)=2\pi/n$, giving $b^{-1}+a^{-1}=1\bmod
n$. Multiplying by $ab$ produces condition (iii) of a data set.

For $1\leq i\leq m$, the preimage of $x_i$ consists of $n/n_i$ points
cyclically permuted by $t$. Each of the points has stabilizer generated by
$t^{n/n_i}$. The rotation angle of $t^{n/n_i}$ must be the same at all
points of the orbit, since its action at one point is conjugate by a power
of $t$ to its action at each other point. So the rotation angle at each
point is of the form $2\pi c_i'/n_i$, where $\gcd(c_i',n_i)$, and as
before, lifting $\gamma_i$ shows that $\rho(\gamma_i)= (t^{n/n_i})^{c_i}$
where $c_i= (c_i')^{-1}\bmod n_i$.

Finally, we have $\rho(\prod_{j=1}^{g_0}[a_j,b_j])=1$, since $C_n$ is
abelian, so 
\[1=\rho(\alpha\beta\gamma_1\cdots \gamma_m)= 
t^{a+b+(n/n_1)c_1+\cdots +(n/n_m)c_m}\ ,\]
giving condition (iv) of a data set.

The fact that the genus of the data set equals $g$ follows
from the multiplicativity of the orbifold Euler characteristic for
the orbifold covering $F\to \O$:
\[ (2-2g)/n = 2 - 2g_0 + 2\Big(\frac{1}{n}-1\Big) + \sum_{i=1}^m \Big(\frac{1}{n_i}
-1\Big)\ .\] 
Thus $h$ leads to a data set of degree $n$ and genus~$g$.

Suppose now that $h$ interchanges the sides of $C$. Its degree must be even,
and we will write it as $2n$. The points $P$ and $Q$ are now interchanged
by $t$, while $h^2$ is a root of $t_{g+1}$ of order $n$ that does not
interchange the sides. In particular, $n$ must be odd. We assume for now
that $n\geq 3$.

Let $D_P$ and $D_Q$ be the disks centered at $P$ and $Q$, for which
$D_P\cup D_Q=\overline{F-F_0}$. The actions of $t^2$ at $P$ and $Q$ are
conjugate, by $t$, so there exists an equivariant homeomorphism from
$D_Q\cup D_P$ to $D^2\times \{-1,1\}$, where the latter has the action
$t^2(x,-1)=(\exp(2\pi i k/n)\,x,-1)$ and $t^2(x,1)=(\exp(-2\pi i
k/n)\,x,1)$ (the minus sign is not necessary, but is natural for our
construction). We think of $P$ and $Q$ as corresponding to $\{0\}\times
\{1\}$ and $\{0\}\times \{-1\}$ respectively.

Since $2\pi k/n$ is twice $2\pi k/(2n)$ and twice $2\pi
(k+n)/(2n)$, we may further assume that the action of $t$ on $D^2\times
\{-1,1\}$ in these coordinates is either $t(x,-1)=(\exp(-2\pi i
k/(2n))\,\overline{x},1)$ and $t(x,1)=(\exp(2\pi i
k/(2n))\,\overline{x},-1)$, or $t(x,-1)=(\exp(-2\pi i
(k+n)/(2n))\,\overline{x},1)$ and $t(x,1)=(\exp(2\pi i
(k+n)/(2n))\,\overline{x},-1)$.

Figure~\ref{fig:noflips} illustrates the effect of $t$ on $\partial N$ for
the first action, in which $t(x,1)=(\exp(2\pi i
k/(2n))\,\overline{x},-1)$. The indicated angles are $2\pi k/(2n)$. If we
extend $h_0$ to $N$ by sending $(x,t)$ to $(\overline{x},1-t)$ followed by
a simple left-hand twist, as in Figure~\ref{fig:noflips}, then the twisting
angle is $2\pi k/n$, and consequently $h^{2n}=t_{g+1}^{2k}$. Other
extensions to $N$ will differ from this by full twists, giving
$h^{2n}=t^{2k+2jn}$ for some integer~$j$. In any case, $h^{2n}$ cannot
equal $t_{g+1}$. For the second action, in which $t(x,1)=(\exp(2\pi i
(k+n)/(2n))\,\overline{x},-1)$, the amount of twisting on $N$ is still
$2\pi k/n$ plus some number of full twists, so again $h^{2n}=t^{2k+2jn}$.
\par
\begin{figure}
\labellist
\small
\pinlabel $B$ [B] at 165 32
\pinlabel $A$ [B] at 165 177
\pinlabel $t(B)$ [B] at 78 129
\pinlabel $t(A)$ [B] at 77 79
\pinlabel $t^2(B)$ [B] at -17 55
\pinlabel $t^2(A)$ [B] at -17 155
\endlabellist
\centering
\includegraphics[width=23 ex]{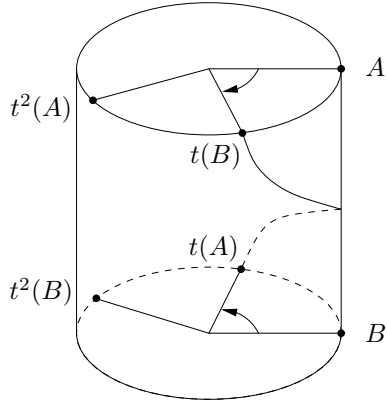}
\caption{An extension of $t$ to $N$ in case $h$ interchanges the sides of
$C$. The amount of left-hand twisting on $N$ is $2\pi k/n$, so
$h^{2n}=t_{g+1}^{2k}$.}
\label{fig:noflips}
\end{figure}

Finally, suppose that $n=1$. Then in the previous construction, $t^2$ is
the identity on $D^2\times\{-1,1\}$, and $t$ is either
$t(x,-1)=(\overline{x},1)$ and $t(x,1)=(\overline{x},-1)$, or
$t(x,-1)=(-\,\overline{x},1)$ and $t(x,1)=(-\,\overline{x},-1)$. In either
case, any extension of $h_0$ to $N$ has some number of full twists, so
$h^2$ is some even power of~$t_C$.

At this point, we have shown how every root of $t_{g+1}$ produces a data
set. If the original roots are conjugate in $\Mod(G)$, then their
restrictions to $F_0$ are conjugate and isotopic to conjugate
homeomorphisms of order $n$, and their extensions to $F$ are conjugate by a
homeomorphism preserving $\{P,Q\}$. Therefore their orbifold quotients $\O$
and $\O'$ are homeomorphic by an orientation-preserving orbifold
homeomorphism preserving taking the distinguished cone points $\{p,q\}$ to
the distinguished cone points $\{p',q'\}$ of $\O'$, and compatible with the
representations of the orbifold fundamental groups to $C_n$. It follows
that our procedure produces equivalent data sets.

Given a data set, we can reverse the argument to produce the root $h$. We
construct the corresponding orbifold $\O$ and representation $\rho\colon
\pi_1^{orb}(\O)\to C_n$. Any finite subgroup of $\pi_1^{orb}(\O)$ is
conjugate to a subgroup of one of the cyclic subgroups generated by
$\alpha$, $\beta$, or a $\gamma_i$, so condition (ii) ensures that the
kernel of $\rho$ is torsionfree. Therefore the orbifold covering $F\to \O$
corresponding to the kernel is a manifold, and calculation of its Euler
characteristic shows that $F$ has genus $g$. Removing disks around the
fixed points $P$ and $Q$ corresponding to the cone points $p$ and $q$
produces the surface $F_0$, and attaching an annulus $N$ produces the
surface $G$ of genus $g+1$. Condition (iii) ensures that the rotation
angles work correctly to allow an extension of $t\vert_{F_0}$ to an $h$
with $h^n$ a single Dehn twist. 

It remains to show that the resulting root of $t_{g+1}$ is determined up to
conjugacy. Our data sets encode the fixed-point data of the periodic
transformation $t$, and it was proven by J. Nielsen~\cite{N1} that this
data determines $t$ up to conjugacy. We require in addition that the
conjugating homeomophism preserve $\{P,Q\}$.

Suppose that $h$ and $h'$ are roots obtained by applying our procedure to a
data set $(n,g_0,(a,b); (c_1,n_1),\ldots,(c_m,n_m))$. That is, we use the
data set to define orbifolds $\O$ and $\O'$ and homomorphisms $\rho\colon
\pi_1^{orb}(\O)\to C_n$ and $\rho'\colon \pi_1^{orb}(\O')\to C_n$, then
take the corresponding covers $F$ and $F'$ and so on. Each of $\O$ and
$\O'$ has genus $g_0$ and $m+2$ cone points of corresponding orders,
including the two distinguished order-$n$ cone points, which give elements
$\alpha$ and $\beta$ and $\pi_1^{orb}(\O)$ and $\alpha'$ and $\beta'$ and
$\pi_1^{orb}(\O')$. We have $\rho(\alpha)=\rho'(\alpha')=t^a$, where the
rotation angles of $t$ and $t'$ at $P$ are $2\pi a^{-1}/n$, and similarly
for the other generators coming from cone points.

We claim that the generators $a_i$ and $b_i$ of $\pi_1^{orb}(\O)$ may
be selected so that $\rho(a_i)=\rho(b_i)=1$ for all $i$. Suppose this is
not initially the case. There is an orbifold homeomorphism of $\O$ whose
effect on the abelianization of $\pi_1^{orb}(\O)$ is to send
$\overline{a_1}$ to $\overline{a_1}\overline{\alpha}$ and to fix the other
generators; it is the end map of an isotopy that slides the cone point $p$
around a loop that represents $\overline{b_1}$. Since $\rho(\alpha)$ is a
generator of $C_n$, we may repeat this homeomorphism some number of times
until for the new $a_1$, $\rho(\overline{a_1})=1$. Repeating this process
on the other $a_i$ and $b_i$, we obtain a new set of generators that verify
the claim.

Performing a similar process, we may assume that
$\rho'(a_i')=\rho'(b_i')=1$ for all~$i$. Now, we take an
orientation-preserving orbifold homeomorphism $k\colon \O\to\O'$ such that
$k_\#(\alpha)=\alpha'$ and so on. It satisfies $\rho'\circ k_\#=\rho$, so
$k$ lifts to a homeomorphism $K\colon F\to F'$ such that $KtK^{-1}=t'$. If
we select $k$ with a bit of care, $K$ carries $F_0$ to $F_0'$, and we can
extend $K|_{F_0}$ to a homeomorphism of $G$ conjugating $h$ to~$h'$.
\end{proof}

Theorem~\ref{thm:main} tells us that $t_{g+1}$ always has a cube root when
$g\geq 1$, corresponding to the data sets $(3,0,(2,2);(c_1,3),\ldots,
(c_g,3))$ with the $c_i$ selected to achieve condition~(iv).
Also, if $t_{g+1}$ has a root of degree $n$, then replacing $g_0$ by
$g_0+1$ in a corresponding data set produces a root of degree $n$ for
$t_{g+n+1}$.

Of more interest is the following:
\begin{corollary}
Suppose that $t_{g+1}$ has a root of degree $n$. Then
\begin{enumerate}
\item[(a)] $n$ is odd.
\item[(b)] $n\leq 2g+1$.
\end{enumerate}
\label{coro:maxroot}
\end{corollary}

\begin{proof}
Part (a) is simply the fact that data sets must have odd degree. For (b),
suppose for contradiction that $n>2g+1$. From the formula for $g$, we have
$1>(2g+1)/n = 1/n+2g_0+\sum_{i=1}^m (1-1/n_i)$ so $g_0=0$, $m=1$, and
$n_1<n$. Putting $d=n/n_1$, condition (iv) gives $a+b= 0\bmod d$,
contradicting condition (iii) since $1<d$ and $d$ divides $n$.
\end{proof}

It may be of interest to note that the maximum degree of a root is half of
the maximum order $4g+2$ of a periodic homeomorphism of $F$, found
by A. Wiman~\cite{Wiman} and W. Harvey~\cite{Harvey}.

\section{The Margalit-Schleimer roots}
\label{sec:Margalit-Schleimer}

Here we will describe the examples of Margalit and Schleimer from our
viewpoint. They construct the surface $F$ by identifying opposite faces of
a $(4g+2)$-gon. It center point is $X_1$, and the two points that come from
identifying vertices are $P$ and $Q$. Pictures centered at $X_1$, $P$, and
$Q$ are shown in Figure~\ref{fig:MS} for the case of $g=2$; in general
$e_4$ becomes $e_{2g}$, and so on. Let $f$ be the homeomorphism of $F$
obtained by rotating through a (counterclockwise) angle of $2\pi/(2g+1)$ at
$P$ and $Q$. It carries $e_0$ to $e_1$, so it rotates through an angle of
$2\pi g/(2g+1)$ at $X_1$. Let $t$ be $f^{-g}$, which rotates through $2\pi
(g+1)/(2g+1)$ at $P$ and $Q$ and through $-2\pi g^2/(2g+1)$ at
$X_1$. Modulo $2g+1$, $-g^2$ is $g/2$ if $g$ is even and $-(g+1)/2$ if $g$
is odd, so $t$ is approximately a quarter turn at $X_1$, counterclockwise
if $g$ is even and clockwise if not. The examples are then obtained by the
construction in Theorem~\ref{thm:main}.
\begin{figure}
\labellist
\small
\pinlabel $X_1$ [B] at 240 120
\pinlabel $P$ [B] at 750 120
\pinlabel $Q$ [B] at 1280 120
\pinlabel $Q$ [B] at 140 415
\pinlabel $P$ [B] at 280 415
\pinlabel $X_1$ [B] at 660 415
\pinlabel $Q$ [B] at 800 415
\pinlabel $X_1$ [B] at 1180 415
\pinlabel $P$ [B] at 1320 415
\pinlabel $P$ [B] at 15 315
\pinlabel $Q$ [B] at 410 315
\pinlabel $Q$ [B] at 535 315
\pinlabel $X_1$ [B] at 940 315
\pinlabel $P$ [B] at 1065 315
\pinlabel $X_1$ [B] at 1465 315
\pinlabel $Q$ [B] at -28 190
\pinlabel $P$ [B] at 440 190
\pinlabel $X_1$ [B] at 495 190
\pinlabel $Q$ [B] at 965 190
\pinlabel $X_1$ [B] at 1015 190
\pinlabel $P$ [B] at 1490 190
\pinlabel $P$ [B] at 10 60
\pinlabel $Q$ [B] at 405 60
\pinlabel $Q$ [B] at 533 60
\pinlabel $X_1$ [B] at 935 60
\pinlabel $P$ [B] at 1060 60
\pinlabel $X_1$ [B] at 1460 60
\pinlabel $Q$ [B] at 140 -40
\pinlabel $P$ [B] at 280 -40
\pinlabel $X_1$ [B] at 660 -40
\pinlabel $Q$ [B] at 800 -40
\pinlabel $X_1$ [B] at 1180 -40
\pinlabel $P$ [B] at 1320 -40
\pinlabel $e_0$ [B] at 215 415
\pinlabel $e_4$ [B] at 65 370
\pinlabel $e_1$ [B] at 350 370
\pinlabel $e_3$ [B] at -10 260
\pinlabel $e_2$ [B] at 430 260
\pinlabel $e_2$ [B] at -10 140
\pinlabel $e_3$ [B] at 430 140
\pinlabel $e_1$ [B] at 65 20
\pinlabel $e_4$ [B] at 350 20
\pinlabel $e_0$ [B] at 215 -30
\pinlabel $e_0$ [B] at 880 215
\pinlabel $e_0$ [B] at 1390 215
\pinlabel $e_1$ [B] at 752 340
\pinlabel $e_1$ [B] at 1277 340
\pinlabel $e_2$ [B] at 600 260
\pinlabel $e_2$ [B] at 1115 260
\pinlabel $e_3$ [B] at 628 85
\pinlabel $e_3$ [B] at 1150 85
\pinlabel $e_4$ [B] at 812 55
\pinlabel $e_4$ [B] at 1335 55
\endlabellist
\centering
\includegraphics[width=\textwidth]{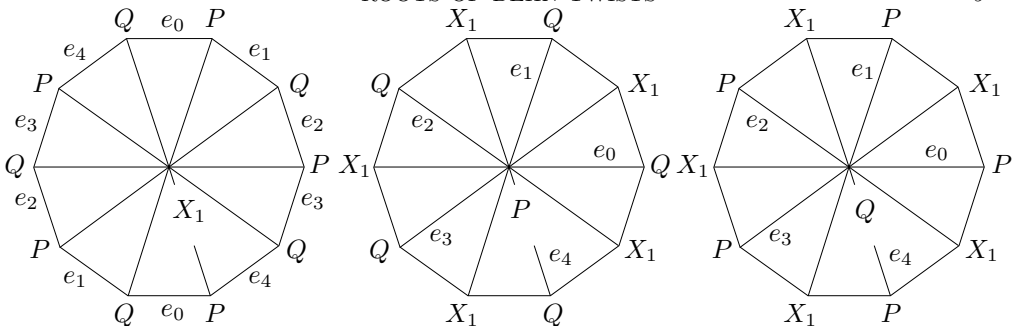}
\caption{The configurations of the edges $e_i$ around $P$ and $Q$, as seen
from $X_1$, $P$, and $Q$ in the Margalit-Schleimer example for $g=2$.}
\label{fig:MS}
\end{figure}

The inverse of $g+1\bmod 2g+1$ is $2$, so $a=b=2$, while the inverse of
$-g^2$ is $c_1=-4$. So the data set resulting from the Margalit-Schleimer
construction is $(2g+1,0,(2,2);(-4,2g+1))$.

We call a root of $t_{g+1}$ a \textit{Margalit-Schleimer} root if it has
degree $2g+1$. Using Theorem~\ref{thm:main}, is is easy to find all the
Margalit-Schleimer roots. We need only find the $x\bmod n$ such that $x$
and $1-x$ are both relatively prime to $n$, then put $a=x^{-1}$ and
$b=(1-x)^{-1}\bmod n$, and $c_1=-a-b\bmod n$. A GAP function to list such
roots is provided in the software at~\cite{software}. For example, we find
that $t_{11}$ has three Margalit-Schleimer roots, $( 21, 0, ( 2, 2 );( 17,
21 ))$, $( 21, 0, ( 5, 17 );( 20, 21 ) )$, and $( 21, 0, ( 11, 20 ); ( 11,
21 ) )$, and $t_{1,001}$ has~$284$.

\section{Genus sets}
\label{sec:genus_set}

The \textit{genus set} of $n$, $g(n)$, is the set of $g$ such that
$t_{g+1}$ has a root of degree $n$. Corollary~\ref{coro:maxroot} tells us
that $g(n)$ is empty for even $n$. For odd $n$, we can gain considerable
information about the genus set. For the rest of this section, $n$ will be
assumed odd, and $n_0$ will denote $(n-1)/2$.

A data set with all $n_i=n$ is called a \textit{primary} data set, and the
corresponding root of $t_{g+1}$ is called a \textit{primary root.}
Primary data sets exist for all $m\geq 1$, since we may take $a=b=2$, and
the $c_i$ selected from ${-4,-2,1,-1}$ so that $a+b+\sum c_i= 0\bmod n$.

We now examine the genera of primary data sets. A quick example will make
it much easier to follow the notation. For $n=9$, so that $n_0=4$, we
position the genera according to their values $\bmod\ n_0$:
\begin{small}\[
\begin{tabular}{cccc}
(0) & (1) & (2) & (3) \\
4 & (5) & (6) & (7) \\
8 & (9) & (10) & (11) \\
12 & 13 & (14) & (15)  \\
16 & 17 & (18) & (19)\\
20 & 21 & 22 & (23) \\
24 & 25 & 26 & (27) \\
28 & 29 & 30 & 31 \\
\end{tabular}
\]
\end{small}\noindent The genus of a primary data set
$(n,g_0,(a,b);(c_1,n),\ldots,(c_m,n))$ is $ng_0+mn_0$.  For $g_0=0$, we
obtain the values $mn_0$ for $m\geq 1$, which for $n=9$ are the values in
the first column other than $0$. For $g_0=1$, $n+mn_0$ is always $1\bmod
n_0$, and we obtain all values greater than $n$. Similarly, $g_0=2$
gives the values in the third column greater than $2n$, and $g_0=3$ gives
those in the last column beyond $3n$. Higher values of $g_0$ give no new
values for $g$. So the primary data sets for $n=9$ give all values of $g$
except the $16$ values indicated in the table.

In general, the genera obtained from data sets of degree $n$ having $0\leq
g_0 <n_0$ are $g_0\bmod n_0$, and are exactly those with $g>g_0n$. No
new genera are obtained when $g_0\geq n_0$. So the genera not obtained
are those in the ``triangular'' set $T(n)$ defined by
\begin{gather*}T(n) = \cup_{0\leq g_0<n_0}T_{g_0}(n),\text{\ where}\\
T_{g_0}(n) =\{ g_0+mn_0\;\vert\; 0\leq m\leq 2g_0\}\ .
\end{gather*}
\noindent Since $T_{g_0}(n)$ has $2g_0+1$ elements, $T(n)$ has $n_0^2$
elements. The maximum element in $T(n)$ is the maximum element in
$T_{n_0-1}(n)$, which is $(n_0-1)n=n(n-3)/2$.

Since the primary data sets produce roots for every genus other than those
in $T(n)$, we have
\begin{corollary} For $n$ odd, $g(n)$ contains all $g\geq 0$ that are
not in $T(n)$. Consequently, $t_{g+1}$ has a root of degree $n$ whenever
$g+1> (n-2)(n-1)/2$.\par
\label{coro:almostall}
\end{corollary}
When $n$ is prime, all data sets are primary. So we have
\begin{corollary} For $n$ prime, $g(n)$ equals the set of $g$
not in $T(n)$. In particular, $t_{(n-2)(n-1)/2}$ does not have a root of
degree $n$.
\label{coro:primecase}
\end{corollary}
\noindent 
For example, $t_{g+1}$ has a cube root for all $g+1\geq 2$, and a fifth
root exactly when $g+1$ is not $1$, $2$, $4$, or $6$. For $n$ that are not
prime, determination of $g(n)$ is more complicated, as elements in $T(n)$
often arise from non-primary data sets. For example, $7\notin T(9)$, but a
ninth root for $t_8$ arises from the data set $(9,0,(2,2);(2,9),(1,3))$,
for which condition~(iv) is satisfied since $a+b+c_1+(9/3)c_2= 0\bmod 9$.

We note that $T_n$ contains about half of the values with $g\leq
n(n-3)/2$. Therefore in Figure~\ref{fig:rootset}, the pairs $(g,n)$
corresponding to primary roots would be about half of the pairs with $n$
odd in the region above $g=(n-1)(n-2)/2$.

\section{The root set and $(d,e)$-roots}
\label{sec:de}

The \textit{root set} $R(g)$ is the set of $n$ such that $t_{g+1}$ has a
root of degree $n$ (although \textit{degree set} would be a more accurate
name). Corollary~\ref{coro:primecase} allows us to effectively compute the
primes in $R(g)$. From Corollary~\ref{coro:almostall}, $R(g)$ contains $n$
whenever $(n-2)(n-1)/2\leq g$. In Theorem~\ref{thm:largeroots}, we will
determine all $n$ in $R(g)$ that satisfy $n\geq g$. First, we must
introduce $(d,e)$-roots. 

Let $d$ and $e$ be odd integers with $d,e\geq 3$. A root corresponding to a
data set having $g_0=0$, $m=2$, $n_1=d$ and $n_2=e$ is called a
\textit{$(d,e)$-root.} The next lemma requires an elementary
number-theoretic fact for which we are unable to find a reference.
To avoid interruption of the argument here, we will prove it later
as Lemma~\ref{lem:relatively_prime}.

\begin{lemma} For any odd integers $d,e\geq 3$, there exist $(d,e)$-roots. 
Such roots satisfy the following:
\begin{enumerate}
\item[(a)] $n=\dfrac{de}{\gcd(d,e)}$, i.~e.~$n=\lcm(d,e)$.
\item[(b)] $g=n-\dfrac{d+e}{2\gcd(d,e)}=n\Big(1-\dfrac{1}{2d}-\dfrac{1}{2e}\Big)$.
\item[(c)] $g+1\leq n< 6(g+2)/5$.
\item[(d)] $n=g+1$ exactly when $d=e=n$.
\end{enumerate}
\label{lem:DEroots}
\end{lemma}

\noindent For example, a $(3,5)$-root has $(g,n)=(11,15)$ (so
Lemma~\ref{lem:DEroots}(c) is best possible, in general), and for $n=105$
there are $(3,35)$-roots when $g=86$ and $(15,7)$-roots when $g=94$. For
even $g$, there is always a $(g+1,g+1)$-root given by
$(g+1,(2,2);(-2,g+1),(-2,g+1))$.

\begin{proof}[Proof of Lemma~\ref{lem:DEroots}]
Put $d_0=\gcd(d,e)$ and $n=de/d_0$, so $\gcd(n/d,n/e)=1$. Let
$n_1=d$, $n_2=e$, and $a=b=2$. Condition (iv) becomes
$4+c_1(n/d)+c_2(n/e)=0$. Since $\gcd(n/d,n/e)=1$, we can write $\ell_1(n/d)
+ \ell_2 (n/e)=1$, and by Lemma~\ref{lem:relatively_prime} we may assume that
$\gcd(\ell_1,d)=\gcd(\ell_2,e)=1$. Taking $c_1=-4\ell_1$ and $c_2=-4\ell_2$
satisfies condition (iv). The genus works out to be the expressions in (b),
which imply the first inequality in~(c). For the second, we have
\[n=g+\frac{d+e}{2d_0}\leq
g+\frac{3+\frac{de}{3}}{2d_0}=
g+\frac{3}{2d_0}+\frac{1}{6}\frac{de}{d_0}<g+2+\frac{n}{6}\ .\] 
Part (d) follows because (b) gives $n=g+1$ when $d=e=n$, and when $d\neq
e$, $\dfrac{d+e}{2d_0}>1$ so $g+1<n$.
\end{proof}

For a given $n$ one can easily compute the $g$ for which $n$ is a
$(d,e)$-root of $g$, if we have a prime factorization of $n$. For in (b) of
Lemma~\ref{lem:DEroots}, $n/d$ and $n/e$ are relatively prime divisors of
$n$. For each pair $(d_1,d_2)$ of relatively prime divisors, we write
$n=d_0d_1d_2$ and put $d=d_0d_1$ and $e=d_0d_2$ giving $n$ as a
$(d,e)$-root for $g=n-(d_1+d_2)/2$ by Lemma~\ref{lem:DEroots}(b). This
gives an algorithm for computing the $(d,e)$-roots of $g$, again assuming
that we can factor, just by checking which of the $n$ in the range allowed
by Lemma~\ref{lem:DEroots}(c) have $g$ among its corresponding genera.  We
have implemented these algorithms as a GAP script \cite{GAP} available
at~\cite{software}. Some sample calculations include the genera having a
$(d,e)$-root of degree $n$:
\smallskip

\noindent \texttt{gap$>$ DERootGenera( 54573 );}\\
\texttt{[ 45476, 45477, 54571, 54572 ]}
\smallskip

\noindent and all $(d,e)$-roots for a given genus:
\smallskip

\noindent \texttt{gap$>$ DERoots( 54572 );}\\
\texttt{[ 54573, 54575, 54587, 54769, 65487 ]}\\
\texttt{gap$>$ DERoots( 54573 );}\\
\texttt{[  ]}
\smallskip

The main result of this section describes all roots of large degree:
\begin{theorem}
Suppose $t_{g+1}$ has a root of degree $n\geq g$. Then the root is either a
Margalit-Schleimer root, a $(d,e)$-root, or the cube root of $t_4$.\par
\label{thm:largeroots}
\end{theorem}

\begin{proof}
Since $n\geq g$, we have
\[ 1 \geq \frac{g}{n} = g_0 + \frac{1}{2}
\sum_{i=1}^m\Big(1-\frac{1}{n_i}\Big) \geq 2g_0 + \frac{m}{3}\ .\]
Therefore $g_0=0$ and $m\leq 3$.

Suppose first that $m=1$. We cannot have $n_1<n$, for putting $d=n/n_1$,
condition (iv) would say that $a+b+dc_1= 0\bmod n$, which is impossible
since $a+b$ is relatively prime to $n$ and hence to $d$. So $n_1=g$, and
$h$ is a Margalit-Schleimer root.

If $m=2$, then $h$ is a $(d,e)$-root.

Suppose that $m = 3$. From our expression for $\frac{g}{n}$, we find that
$1\leq \frac{1}{n_1}+\frac{1}{n_2}+\frac{1}{n_3}$. Since all $n_i$ are odd
this can only be satisfied when $n_1=n_2=n_3=3$. Condition (iv) says that
$a+b = 0\bmod n/3$, a contradiction unless $n=3$. That is, $h$ is a cube
root with $g=3$. In fact, this $h$ is unique, since the only data set of
degree $3$ and genus $3$ is $(3,0,(2,2);(1,3),(2,3),(2,3))$.
\end{proof}

In view of Lemma~\ref{lem:DEroots}(c), Theorem~\ref{thm:largeroots} has the
amusing consequence that the only $t_{g+1}$ which has a root of degree $g$
is~$t_4$.

\section{There are no orientation-reversing roots}
\label{sec:orientation-reversing}

In this section, we will prove that $t_{g+1}$ has no orientation-reversing
roots, and that roots of $t_{g+1}$ cannot be conjugate by
orientation-reversing homeomorphisms. Consequently, Theorem~\ref{thm:main}
classifies all roots of $t_{g+1}$ in the homeomorphism group, up to
conjugacy.

\begin{proposition} Let $h$ be an orientation-reversing homeomorphism of
$G$ with $h^n$ isotopic to $t_{g+1}^\ell$ for some $n>0$. Then $\ell = 0$.
\label{prop:orientation-reversing}
\end{proposition}
\begin{proof}
As in the proof of Theorem~\ref{thm:main}, we write $t_{g+1}=t_C$ and
change $h$ by isotopy so that $h$ restricts to a homeomorphism $h_0$ of
finite order on $\overline{G-N}$ for some annulus neighborhood $N$ of
$C$. On $N$, $h$ is orientation-reversing and has finite order on $\partial
N$.

Suppose first that $h$ preserves the components of $\partial N$. Then $h$
reverses orientation on each component, so is a reflection of period
$2$. It follows that $h_0$ has order $2$, and for some coordinates on $N$
as $S^1\times I$, $h$ is isotopic to a homeomorphism of the form
$(z,t)\mapsto (\overline{z},t)$. Therefore $h^2$ is isotopic to the
identity on $G$, so $\ell=0$.

Suppose now that $h$ interchanges the components of $\partial N$. Since $h$
is orientation-reversing and has finite order on $\partial N$, there are
coordinates on $\partial N$ as $S^1\times \{0,1\}$ so that $h(z,t)=(e^{2\pi
k/n}z,1-t)$. Let $e$ be the homeomorphism of $N$ defined by
$(z,t)\mapsto(e^{2\pi k/n}z,1-t)$. Then $h|_N$ is isotopic relative to
$\partial N$ to $t_C^re$ for some power $r$. Since $et_C$ is isotopic to
$t_C^{-1}e$ relative to $\partial N$, and $n$ must be even, $h^n$ is
isotopic to the identity, that is, $\ell=0$.
\end{proof}

We note also that no two roots of $t_C$ can be conjugate by an
orientation-reversing homeomorphism. For if $h_1$ and $h_2$ are roots and
$gh_1g^{-1}=h_2$ with $g$ orientation-reversing, then $gt_Cg^{-1}=t_C$. But
conjugation of a left-handed Dehn twist by an orientation-reversing
homeomorphism produces a right-handed Dehn twist.

\section{Roots of $t^\ell$}
\label{sec:fractional}

Theorem~\ref{thm:main} gives some information about the roots of powers of
$t_{g+1}$, that is, the fractional powers of $t_{g+1}$. A tuple like a data
set except that condition (iii) is replaced by the condition that $a+b
=\ell ab\bmod n$ produces a root of $t_{g+1}^\ell$ of degree $n$. The only
difference in the construction is that the rotation angles at $P$ and $Q$
are of the form $2\pi k/n$ and $2\pi (\ell-k)/n$, and the twisting on the
annulus $N$ is through an angle $2\pi \ell/n$ rather than $2\pi/n$. Thus
the data set $(4,0,(1,1);(1,2))$ for which $a+b=2ab\bmod 4$ yields a
root of $t_2^2$ of degree $4$. Of course we know from
Corollary~\ref{coro:maxroot} that $t_2$ does not have a square root. The
data set $(3,0,(1,1);(2,3),(2,3))$ gives a cube root of $t_3^2$.

There are some complications, however. If $\ell$ and $n$ are not relatively
prime, then a root of degree $n$ of $t_{g+1}^\ell$ might be a power of a
root of a smaller power of $t_{g+1}$ of lower degree, and then the action
on $F$ in the proof of Theorem~\ref{thm:main} will not be effective. More
interesting is the fact that roots of $t_{g+1}^\ell$ may exchange the sides
of $C$, requiring a different kind of quotient orbifold to be analyzed.

\section{An elementary lemma}
\label{sec:lemma}

In Section~\ref{sec:de} we needed an elementary number-theoretic fact,
Lemma~\ref{lem:relatively_prime}. We are grateful to Ralf Schmidt for
providing us with a much better proof than our original one.
\begin{lemma} Let $d_1$, $d_2$ be relatively prime positive integers, and
let $Q$ be a finite set of primes. If $2\in Q$, assume that $d_1$ and $d_2$
are not both odd. Then there exist integers $c_1$ and $c_2$ so that
$c_1d_1+c_2d_2=1$ and neither $c_1$ nor $c_2$ is divisible by any prime
in~$Q$.\par
\label{lem:relatively_prime}
\end{lemma}
\begin{proof}
Choose $A$ and $B$ with $Ad_2+Bd_1=1$, so that $(A-kd_1)d_2+(B+kd_2)d_1=1$
for all integers $k$. We seek a $k$ so that $A-kd_1$ and $B+kd_2$ are
nonzero $\bmod q_i$ for each $q_i\in Q$.

For each odd $q_i\in Q$, if any $A-kd_1=0\bmod q_i$, then $\gcd(q_i,d_1)=1$.
So there is a unique $k_i\bmod q_i$ such that $A-kd_1=0\bmod q_i$ exactly
when $k=k_i\bmod q_i$. Similarly, if any $B+kd_2=0\bmod q_i$, then such $k$
are those with $k=\ell_i\bmod q_i$ for a unique $\ell_i\bmod q_i$. Since
$q_i\geq 3$, there are choices of $m_i$ so that if $k=m_i\bmod q_i$, then
neither $A-kd_1=0\bmod q_i$ nor $B+kd_2=0\bmod q_i$. If $q_i=2$, then we
may assume that $d_2$ is even and $d_1$ and hence $B$ are odd, and we take
$m_i$ equal to $0$ or $1$ according as $A$ is odd or even. The $k$ we are
seeking include all those satisfying $k=m_i\bmod q_i$ for all $i$, and such
$k$ exist by the Chinese Remainder Theorem.
\end{proof}
\longpage

\bibliographystyle{amsplain}

\end{document}